\documentclass[12pt]{article}%
\usepackage{amsmath}
\usepackage{amsfonts}
\usepackage{amssymb}
\usepackage{graphicx}%
\providecommand{\U}[1]{\protect\rule{.1in}{.1in}}
\newtheorem{theorem}{Theorem}

\newtheorem{corollary}[theorem]{Corollary}

\newtheorem{definition}[theorem]{Definition}
\newtheorem{example}[theorem]{Example}

\newtheorem{lemma}[theorem]{Lemma}

\newtheorem{proposition}[theorem]{Proposition}

\newenvironment{proof}[1][Proof]{\noindent\textbf{#1.} }{\ \rule{0.5em}{0.5em}}
\begin{document}

\title{Uniform Integrability in vector lattices and applications.}
\author{Youssef Azouzi\\{\small Research Laboratory of Algebra, Topology, Arithmetic, and Order}\\{\small Department of Mathematics}\\{\small Faculty of Mathematical, Physical and Natural Sciences of Tunis}\\{\small Tunis-El Manar University, 2092-El Manar, Tunisia}}
\date{}
\maketitle

\begin{abstract}
In this paper, we explore an abstraction of uniform integrability in vector
lattices and demonstrate its application by providing a positive solution to
an open question posed by Kuo, Rodda, and Watson. Specifically, we show that
for every $p\in(1,\infty)$ and with $T$ as a conditionally expectation
operator, spaces $L^{p}\left(  T\right)  $ are sequentially complete.
Furthermore, we demonstrate that a de La Vall\'{e} Poussin Theorem does not
hold in the general setting of vector lattices.

\end{abstract}

\section{Introduction}

The concept of uniform integrability is an efficient tool in analysis and in
probability and specially in martingale theory, see for example \cite[Chapers
13,14]{b-23} and \cite{b-1804}. An analogous notion has been introduced by
Kuo, Vardy and Watson in the theory of free-measure framework of Riesz spaces,
as mentioned in \cite{L-04}. In this framework we consider a Dedekind complete
Riesz space $E$ with weak order unit $e,$ equipped with a conditional
expectation operator $T$ that fixes $e.$ This operator is crucial in defining
various mathematical concepts such as filtrations, martingales, stopping
times, and more. Several researchers, including Kuo, Labushagne, Watson,
Grobler, Azouzi, Ramdane, and others, have made significant progress in this
area, as documented in \cite{L-24,L-03, L-14,L-180,L-311,L-06,L-11}. This
setting has allowed for the generalization of many fundamental and advanced
results from the classical theory of probability. However, certain basic
notions, such as completeness, continue to present challenges. While some
theorems hold true in surprisingly general situations with relatively simple
proofs, others are much more difficult to formulate and prove satisfactorily.
The Egorov theorem is a good example of the latter. Formulating a `good'
statement that appears to be true can be a promising start, but in many cases,
proving the statement requires a deep understanding of the subject. To address
difficult problems in the field, new concepts and techniques often need to be
introduced. This work focuses on studying the strong completeness of spaces
$L^{p}\left(  T\right)  $. This concept is analogous to the Riesz-Fisher
Theorem, which asserts that spaces $L^{p}\left(  \mu\right)  $ are
norm-complete.\textrm{ }In the context of Riesz spaces, our spaces are
equipped with a vector-valued norm, and convergence is defined in terms of
order. Therefore, completeness for sequences and nets are different concepts.
The main tool utilized in this work is $T$-uniformity, which is an abstraction
of uniform integrability. This work begins with a brief overview of the
relevant topic, followed by the presentation of our new results and their
applications. Specifically, we demonstrate that if $E$ is a Dedekind complete
Riesz space with a weak order unit $e$, and $T$ is a conditional expectation
operator, then the Riesz space $L^{p}\left(  T\right)  $ (as defined in
Section 2) is sequentially complete for $1<p<\infty.$ Our result resolves an
open question that was raised in \cite{L-360}. The authors of \cite{L-360} had
previously established the strong sequential completeness of the space
$L^{1}\left(  T\right)  .$ Their proof is not an adaptation of the proof for
classical space $L^{1}\left(  \mu\right)  $ because it is no longer true that
a strong Cauchy sequence $\left(  x_{n}\right)  $ has a subsequence $\left(
y_{n}\right)  $ such that the series $\sum y_{n}$ is order convergent. Our
result completes their work and establishes that all spaces $L^{p}\left(
T\right)  $ are squentially strong complete for $p\in\left[  1,\infty\right]
$. However, the full completeness of these spaces is still an open problem,
and has only been proven for two values of $p$. The first case, $p=\infty,$ as
a relatively straightforward proof and can be found in \cite{L-360}, while the
second case (for $p=2$) is much more complex and has recently been
demonstrated by Kalauch, Kuo, and Watson in \cite{L-928}. Their proof is based
on a Hahn-Jordan type decomposition which is proved in a separate paper
\cite{L-886}. For informations about spaces $L^{p}\left(  T\right)  $ we refer
the reader to \cite{L-180}.

Several forms of convergence will be employed in this work, none of which are
topological. We will primarily focus on four types of convergence: order
convergence, strong convergence, and their unbounded counterparts, which are
analogous to almost sure convergence and convergence in probability.
Throughout our presentation, we will strive to elucidate the motivations
behind these generalizations and how they can help to provide a more profound
understanding of the crucial role played by order in probability theory and,
more broadly, in measure theory.

For more information about order convergence and unbounded order convergence
the reader is referred to \cite{L-65,L-444} and references therein.
Additionally, for a discussion of convergence in $T$-conditional probability,
we refer the reader to \cite{L-14}.

The paper is organized as follows. In Section 2, we provide some
preliminaries, including basic properties of convergence in $T$-conditional probability.

Section 3 focuses on $T$-uniformity, the central tool in our work. We
demonstrate that the de La Vall\'{e}e Poussin criterion for uniform
integrability fails in the setting of Riesz spaces. Moreover, we establish
that the condition given by this criterion is only sufficient and no longer
necessary for a family in $L^{1}\left(  T\right)  $ to be $T$-uniform. We also
prove a version of Vitali's theorem in $L^{p}\left(  T\right)  ,$ which is a
crucial result in demonstrating the strong sequential completeness of the
spaces $L^{p}\left(  T\right)  .$ In the last section, we apply our earlier
results to make significant progress in completeness by showing that all
spaces $L^{p}\left(  T\right)  $ are sequentially strong complete. This result
answers a question posed by Kuo, Rodda and Watson in \cite{L-360} as mentioned earlier.

\section{Preliminaries}

In the following discussion, we will be working with a Riesz conditional
triple $\left(  E,e,T\right)  ,$ where $E$ is a Dedekind complete Riesz space
(or vector lattice), $e$ a weak order unit and $T$ is a conditional
expectation operator, that is, an order continuous strictly positive
projection whose range $R\left(  T\right)  $ is a Dedekind complete Riesz
subspace with $Te=e.$ The universal completion of $E$ is denoted by $E^{u},$
while its sup-completion is denoted by $E^{s}.$ It is shown in \cite{L-24}
that there exists a largest Riesz subspace of $E^{u}$, known as the
\textbf{natural domain} of $T,$ denoted by $L^{1}(T),$ to which $T$ extends
uniquely to a conditional expectation. It is noteworthy that $L^{1}\left(
T\right)  $ possesses a desirable characteristic of being $T$-universally
complete. This means that any increasing net $(x_{\alpha})$ in $L^{1}(T)^{+}$
with $\left(  Tx_{\alpha}\right)  $ order bounded in $L^{1}(T)^{u}$ converges
in order. Accordingly, once we extend the conditional expectation $T$ on $E$
to its natural domain $L^{1}(T)$, the space $E$ becomes irrelevant. This
eliminates any reason for considering conditional expectations on other than
their natural domains. Now, it is well-known that a multiplication can be
defined in the universal completion $L^{1}(T)^{u}$ of $L^{1}(T)$ so that
$L^{1}(T)^{u}$ becomes an $f$-algebra with $e$ as identity (see \cite[Section
50]{b-241} for more details. Spaces $L^{p}\left(  T\right)  $ are introduced
in \cite{L-180} as a generalization of $L^{2}\left(  T\right)  $ introduced
earlier by Labuschagne and Watson. For $p\in\left(  1,\infty\right)  $ we put%

\[
L^{p}(T)=\{x\in L^{1}(T):|x|^{p}\in L^{1}(T)\},
\]
and
\[
N_{p}(x)=T(|x|^{p})^{1/p}\text{ for all }x\in L^{p}(T).
\]

The space $L^{\infty}\left(  T\right)  $ is defined as the ideal of
$L^{1}\left(  T\right)  ^{u}$ generated by $R\left(  T\right)  :$%
\[
L^{\infty}\left(  T\right)  =\left\{  x\in L^{1}\left(  T\right)  :\left\vert
f\right\vert \leq u\text{ for some }u\in R\left(  T\right)  \right\}  .
\]
It is equipped with the $R\left(  T\right)  $-vector valued norm $\left\Vert
.\right\Vert _{T,\infty}$ given by%
\[
N_{\infty}\left(  x\right)  =\inf\left\{  u\in R\left(  T\right)  :\left\vert
x\right\vert \leq u\right\}  \text{ for all }x\in L^{\infty}\left(  T\right)
.
\]

For further details about these spaces, we suggest referring to \cite{L-180}.
In the context of Riesz spaces, the notion of convergence almost surely has
been extended to unbounded order convergence. Interestingly, in $L^{p}$-spaces
for sequences, these two types of convergence coincide. Furthermore,
convergence in probability (or more generally in measure) has been generalized
to convergence in $T$-conditional probability. We would like to remind the
reader of the following result from \cite{L-14}.

\begin{definition}
A net $\left(  f_{\alpha}\right)  $ in $E$ is said to be convergent to $f$ in
$T$-conditional probability if $TP_{\left(  \left\vert f_{\alpha}-f\right\vert
-\varepsilon e\right)  ^{+}}e$ converges to $0$ in order, where $\varepsilon$
is any positive number.
\end{definition}

Let us recall that if $\theta$ is any type of convergence in a Riesz space
then its unbounded version, which can be denoted by $u\theta,$ is defined as
follows: a net $\left(  x_{\alpha}\right)  $ is said to be unbounded $\theta$
converging to $x,$ and we write $x_{\alpha}\overset{\mathfrak{u\theta}%
}{\longrightarrow}x$ if $\left\vert x_{\alpha}-x\right\vert \wedge
u\overset{\theta}{\longrightarrow}0.$ There are three types of convergences
that have been widely studied in this context: order convergence, norm
convergence, and absolute weak convergence. We record next a lemma \cite[Lemma
4]{L-900} which will be used in the proof of the main result of this paper.
The lemma states that the convergence in $T$-conditional probability is the
unbounded version of the $T$-strong convergence, in other words convergence in
$T$-conditional probability is exactly the unbounded norm convergence with
respect to the $R\left(  T\right)  $-vector valued norm $\left\Vert
.\right\Vert _{1,T}.$

\begin{lemma}
\label{UI-B}\textit{Let} $\left(  E,e,T\right)  $ \textit{be a conditional
Riesz triple}. \textit{For a net} $(x_{\alpha})_{\alpha\in A}$ \textit{in}
$E$, \textit{the following statements are equivalent}.

\begin{enumerate}
\item $x_{\alpha}\rightarrow x$ \textit{in }$T$\textit{-conditional
probability};

\item $T(|x_{\alpha}-x|\wedge u)\overset{o}{\longrightarrow}$ \textit{for
every} $u\in E_{+};$

\item $T(|x_{\alpha}-x|\wedge e)\overset{o}{\longrightarrow}0$.
\end{enumerate}
\end{lemma}

In probability theory, it is known that if a sequence $\left(  X_{n}\right)  $
of random variables converges in probability to $X,$ then $X_{n}%
^{k}\longrightarrow X^{k}$ in probability for any $k\in\mathbb{N}.$ This can
be easily deduced from the fact that $\left(  X_{n}\right)  $ converges in
probability to $X$ if and only if every subsequence of $\left(  X_{n}\right)
$ has a further subsequence that converges to $X$ almost surely. While there
is no direct Riesz space version of this result, a different proof can be
provided. The third item in the next lemma will be used in the proof of our
main result.

\begin{lemma}
\label{UI-C}Let $\left(  E,e,T\right)  $ be a conditional Riesz triple and
$\left(  x_{\alpha}\right)  $ a net in $E^{u}$ converging in $T$-conditional
probability to $x.$

\begin{enumerate}
\item We have for all $u\in E,$
\[
T\left(  \left\vert x_{\alpha}\right\vert \wedge u\right)  \overset
{o}{\longrightarrow}T\left(  \left\vert x\right\vert \wedge u\right)  .
\]

\item If $M\in E,$ then $Mx_{\alpha}\longrightarrow Mx$ in $T$-conditionally probability.

\item If $\left(  x_{a}\right)  $ is bounded \textit{then for for every
}$p\geq1,$ $\left\vert x_{\alpha}\right\vert ^{p}\longrightarrow\left\vert
x\right\vert ^{p}$ in $T$-\textit{conditional probability.}
\end{enumerate}
\end{lemma}

\begin{proof}
(i) The desired result follows from the following inequalities:%
\[
\left\vert T\left(  \left\vert x_{\alpha}\right\vert \wedge u\right)
-T\left(  \left\vert x\right\vert \wedge u\right)  \right\vert \leq
T\left\vert \left\vert x_{\alpha}\right\vert \wedge u-\left\vert x\right\vert
\wedge u\right\vert \leq T\left(  \left\vert x_{\alpha}-x\right\vert \wedge
u\right)  .
\]

(ii) Let $A=M\vee e$ and $k\in\mathbb{N}.$ Then%
\begin{align*}
T\left(  M\left\vert x_{\alpha}-x\right\vert \wedge e\right)   &  \leq
T\left(  A\left\vert x_{\alpha}-x\right\vert \wedge A\right)  =TA\left(
\left\vert x_{\alpha}-x\right\vert \wedge e\right) \\
&  =T\left[  \left(  A-A\wedge ke\right)  \left(  \left\vert x_{\alpha
}-x\right\vert \wedge e\right)  \right]  T\left[  \left(  A\wedge ke\right)
\left(  \left\vert x_{\alpha}-x\right\vert \wedge e\right)  \right] \\
&  \leq T\left(  A-A\wedge ke\right)  +kT\left(  \left\vert x_{\alpha
}-x\right\vert \wedge e\right)
\end{align*}
We deuce that
\[
\limsup\limits_{\alpha}T\left(  M\left\vert x_{\alpha}-x\right\vert \wedge
e\right)  \leq T\left(  A-A\wedge ke\right)  .
\]
As this happens for every $k$ we get $\limsup\limits_{\alpha}T\left(
M\left\vert x_{\alpha}-x\right\vert \wedge e\right)  =0,$ which gives the
desired result.

(iii) We will use the following inequality, which is true for all $a,b\in
E_{+},$%
\[
\left\vert a^{p}-b^{p}\right\vert \leq p\left\vert a-b\right\vert \left(
a\vee b\right)  ^{p-1}.
\]
Let $M$ be any positive element satisfying $\left\vert x_{\alpha}\right\vert
\leq M$ for all $\alpha$ and $\left\vert x\right\vert \leq M.$ Then
\[
\left\vert \left\vert x_{\alpha}\right\vert ^{p}-\left\vert x\right\vert
^{p}\right\vert \leq p\left\vert x_{\alpha}-x\right\vert M.
\]
Hence%
\[
T\left\vert \left\vert x_{\alpha}\right\vert ^{p}-\left\vert x\right\vert
^{p}\right\vert \wedge e\leq Tp\left\vert x_{\alpha}-x\right\vert M\wedge e,
\]
and it follows from (ii) that $T\left\vert \left\vert x_{\alpha}\right\vert
^{p}-\left\vert x\right\vert ^{p}\right\vert \wedge e\overset{o}%
{\longrightarrow}0,$ which gives (iii).
\end{proof}

\section{Uniform integrability in vector lattices}

The notion of uniform integrability has been extended to the setting of vector
lattices to $T$-uniformity. The following definition is taken from \cite{L-04}.

\begin{definition}
\label{UI-A}A family $A$\ in $L^{1}\left(  T\right)  $\ is said to be
$T$-\textit{uniform}\ if%
\[
\sup\limits_{x\in A}\{TP_{(|x|-ce)^{+}}|x|\rightarrow0\text{ as }%
c\rightarrow\infty.
\]

\end{definition}

Every $T$-uniform family $A$ is $T$-bounded in $L^{1}\left(  T\right)  $ (see
\cite[Lemma 2.7]{L-04}), meaning that $\sup\limits_{a\in A}T\left\vert
x\right\vert $ exists in $L^{1}\left(  T\right)  .$ However, the converse is
not true. In classical probability theory, we use an alternative definition of
uniform integrability, known as the `epsilon-delta' definition. The following
result is widely recognized as useful.

\begin{proposition}
Let $\left(  \Omega,\mathcal{F},\mathbb{\lambda}\right)  $ be a probability
triple and $\mathcal{A}$ a family of random variables in $L_{1}=L_{1}\left(
\Omega,\mathcal{F},\mathbb{\lambda}\right)  .$ Then the following statements
are equivalent.

\begin{enumerate}
\item[(i)] The family $\mathcal{F}$ is uniform integrable, that is,
$\sup\limits_{X\in\mathcal{A}}\mathbb{E}\left(  \left\vert X\right\vert
\mathbf{1}_{\left\{  \left\vert X\right\vert \geq n\right\}  }\right)
\longrightarrow0$ as $n\longrightarrow\infty;$

\item[(ii)] $\mathcal{A}$ is norm bounded in $L_{1}$ and for every
$\varepsilon>0$ there exists $\delta>0$ such that%
\[
\lambda\left(  A\right)  <\delta\Longrightarrow\mathbb{E}\left(  \left\vert
X\right\vert \mathbf{1}_{A}\right)  <\varepsilon\text{ }\forall X\in
\mathcal{A}.
\]

\end{enumerate}
\end{proposition}

There is no `epsilon-delta' version of uniform integrability in the setting of
Riesz space theory. However, the concept can still be expressed in a different
manner. It is worth noting that property (ii) mentioned earlier can be
reformulated and comprehended in the following way:%
\[
\sup\mathbb{E}\left(  \left\vert X\right\vert \mathbf{1}_{A}\right)
\longrightarrow0\text{ as }\mu\left(  A\right)  \longrightarrow0,
\]
or also,%
\begin{equation}
\sup\mathbb{E}\left(  \left\vert X\right\vert \mathbf{1}_{A}\right)
\longrightarrow0\text{ as }\mathbf{1}_{A}\longrightarrow0,
\end{equation}
where the latter means the convergence is almost surely. It is easy to see
that if $\left(  A_{n}\right)  $ is a sequence of events then the three kinds
of convergence - in probability, in $L^{1}$-norm and almost surely - coincide
for the sequence $\left(  \mathbf{1}_{A_{n}}\right)  $. This means that, in
language of vector lattices, the sequence of band projections $P_{n}%
:Y\longmapsto Y\mathbf{1}_{A_{n}}$ acting on $L_{1}\left(  \Omega
,\mathcal{F},\mathbb{P}\right)  $ is order null. Hence, an epsilon-delta
definition of uniform integrability in the context of vector lattices can be
conveniently reformulated in terms of band projections, as shown in
\cite{L-14}, which is summarized in the following result.

\begin{proposition}
Let $\left(  E,e,T\right)  $ be a Riesz triple and $A$ a subset of $E.$ Then
the following statements are equivalent.

\begin{enumerate}
\item[(i)] The family $\mathcal{A}$ is $T$-uniform;

\item[(ii)] For every net $\left(  P_{\alpha}\right)  $ is an order null net
of band projections we have $\sup\limits_{x\in A}P_{\alpha}\left\vert
x\right\vert \longrightarrow0$ in order.
\end{enumerate}
\end{proposition}

In \cite{L-180}, it was shown that every bounded family in $L^{p}\left(
T\right)  ,$ where $1<p<\infty,$ is $T$-uniform. This result is instrumental
in establishing a vector lattice variant of the Vitali Theorem, which is
essential in proving Theorem \ref{ThB} below. The next theorem furnishes a
sufficient condition for a family to be $T$-uniform.

\begin{theorem}
\label{T1}\textit{Let} $\left(  E,e,T\right)  $ \textit{be a Riesz conditional
triple and }$\mathcal{F}$ be a family of conditional expectation operators
such that $TS=T$ for every $S\in\mathcal{F}$ and $f\in E.$ \textit{Then the
family} $\{Sf:S\in\mathcal{F}\}$ is $T$-uniform.
\end{theorem}

\begin{proof}
By Markov's inequality we have%
\[
P_{\left(  \left\vert Sx\right\vert -\alpha e\right)  ^{+}}e\leq\alpha
^{-1}\left\vert Sx\right\vert \leq\alpha^{-1}S\left\vert x\right\vert .
\]
It follows that%
\[
TP_{\left(  \left\vert Sx\right\vert -\alpha e\right)  ^{+}}e\leq\alpha
^{-1}TS\left\vert x\right\vert \alpha^{-1}=T\left\vert x\right\vert .
\]
Hence%
\[
\sup\limits_{S\in\mathcal{F}}TP_{\left(  \left\vert Sx\right\vert -\alpha
e\right)  ^{+}}e\leq\alpha^{-1}T\left\vert x\right\vert .
\]
Now
\begin{align*}
TP_{\left(  \left\vert Sx\right\vert -\alpha e\right)  ^{+}}\left\vert
Sx\right\vert  &  \leq TP_{\left(  \left\vert Sx\right\vert -\alpha e\right)
^{+}}S\left\vert x\right\vert =TSP_{\left(  \left\vert Sx\right\vert -\alpha
e\right)  ^{+}}\left\vert x\right\vert =TP_{\left(  \left\vert Sx\right\vert
-\alpha e\right)  ^{+}}\left\vert x\right\vert \\
&  =TP_{\left(  \left\vert Sx\right\vert -\alpha e\right)  ^{+}}\left(
P_{\left(  \left\vert x\right\vert -Ke\right)  ^{+}}\left\vert x\right\vert
+P_{\left(  \left\vert x\right\vert -Ke\right)  ^{+}}^{d}\left\vert
x\right\vert \right) \\
&  \leq K^{-1}T\left\vert x\right\vert +KTP_{\left(  \left\vert Sx\right\vert
-\alpha e\right)  ^{+}}e\leq K^{-1}T\left\vert x\right\vert +\dfrac
{KT\left\vert x\right\vert }{\alpha}.
\end{align*}
So%
\[
\limsup\limits_{\alpha\longrightarrow\infty}\sup\limits_{S\in\mathcal{F}%
}TP_{\left(  \left\vert Sx\right\vert -\alpha e\right)  ^{+}}\left\vert
Sx\right\vert \leq K^{-1}T\left\vert x\right\vert .
\]
\newline Hence
\[
\lim\limits_{\alpha\longrightarrow\infty}\sup\limits_{S\in\mathcal{F}%
}TP_{\left(  \left\vert Sx\right\vert -\alpha e\right)  ^{+}}\left\vert
Sx\right\vert =0,
\]
and the result follows.
\end{proof}

De La Vall\'{e} Poussin's classical theorem offers a characterization of
uniform integrability.

\begin{theorem}
\label{ThA}Let $\mathcal{H}$ be a family in $L^{1}\left(  \mathbb{P}\right)
.$ Then the following statements are equivalent:

\begin{enumerate}
\item[(i)] The family $\mathcal{H}$ is uniform integrable.

\item[(ii)] There exists a convex function $\varphi:[0,\infty)\longrightarrow
\lbrack0,\infty)$ such that
\[
\lim\limits_{x\longrightarrow\infty}\dfrac{\varphi\left(  x\right)  }%
{x}=\infty\text{ and }\sup\limits_{X\in\mathcal{H}}\mathbb{E}\left(
\varphi\left(  \left\vert X\right\vert \right)  \right)  <\infty.
\]

\end{enumerate}
\end{theorem}

The next example demonstrates that the implication (i) \U{21d2} (ii) does not
hold in the context of a Riesz space equipped with a conditional expectation operator.

\begin{example}
Let $E$ be the Riesz space consisting of all functions $f:[1,\infty
\lbrack\longrightarrow\mathbb{R}$ such that $f\left(  x\right)  =O\left(
x\right)  $ as $x\longrightarrow\infty.$ This means that the functions in E
grow at most linearly as x approaches infinity. It is easily seen that $E$ is
a Dedekind complete Riesz space with weak order unit $e,$ where $e\left(
t\right)  =1$ for all $t\in\lbrack1,\infty).$ Let $T:E\longrightarrow E$ be
the identity map and $\mathcal{F}=\left\{  u\right\}  $ with $u\left(
x\right)  =x.$ Then $T$ is a conditional expectation and $\mathcal{F}$ is
$T$-uniform. If $\varphi$ is a test function, which means that $\varphi$ is a
convex function on $[1,\infty)$ and $\lim\dfrac{\varphi\left(  x\right)  }%
{x}=\infty$) then $\varphi\left(  u\right)  \in E^{s}\diagdown E.$
\end{example}

It is worth noting that in the previous example $\varphi\left(  u\right)  \in
E^{u},$ the universal completion of $E,$ and even better $\varphi\left(
u\right)  \in L^{1}\left(  T\right)  ,$ the natural domain of $T.$ This means
that condition (ii) is indeed fulfilled if we work in the space $L^{1}\left(
T\right)  .$ One might then expect that Theorem \ref{ThA} holds in the general
situation for $T$-uniform family. The following more involved example shows,
however, that no such expected result can exist.

\begin{example}
Let $A$ denotes the set of all convex functions $\varphi:[0,\infty
)\longrightarrow\lbrack0,\infty)$ satisfying $\lim\limits_{x\longrightarrow
\infty}\dfrac{\varphi\left(  x\right)  }{x}=\infty$ and consider the product
space $E^{A}$ where $E=L^{1}\left[  0,1\right]  .$ Now for each $\varphi\in A$
one can find an element $X_{\varphi}\in L^{1}\left[  0,1\right]  $ such that%
\[
\int_{\left\{  \left\vert X_{\varphi}\right\vert \geq n\right\}  }\left\vert
X_{\varphi}\right\vert d\mu\leq\dfrac{1}{n}\text{ for all }n=1,2,...
\]
and such that $\int\varphi\left(  \left\vert X_{\varphi}\right\vert \right)
d\mu=\infty.$ Now for each $\varphi$ let $Y^{\varphi}\in L^{1}\left(
T\right)  $ be defined as follows: $Y^{\varphi}\left(  \psi\right)
=X_{\varphi}$ if $\psi=\varphi$ and $Y^{\varphi}\left(  \psi\right)  =0$
otherwise. It is easily seen that the family $\left(  Y^{\varphi}\right)
_{\varphi\in A}$ is $T$-uniform. However, there is no function $\varphi$ in
$A$ such that $\sup T\varphi\left(  \left\vert Y\right\vert \right)  $ exists
in $L^{1}\left(  T\right)  .$
\end{example}

In the Riesz space setting, we can establish a positive result regarding
Theorem \ref{ThA}, which can be used to provide a quick proof of
\cite[Corollary 3.10]{L-180}.

\begin{theorem}
Let $\mathcal{F}$ be a family in $L^{1}\left(  T\right)  $ and $\varphi
:[0,\infty)\longrightarrow\lbrack0,\infty)$ be a convex function such that
$\lim\limits_{x\longrightarrow\infty}\dfrac{\varphi\left(  x\right)  }%
{x}=\infty.$ If $\sup T\varphi\left(  x\right)  $ exists in $L^{1}\left(
T\right)  $ then $\mathcal{F}$ is $T$-uniform.
\end{theorem}

\begin{proof}
Pick $k>0$ and choose $\alpha>0$ such that $t>\alpha\implies\varphi\left(
t\right)  \geq kt.$ This means that the function
\[
t\longmapsto\left(  \varphi\left(  t\right)  -kt\right)  \chi_{\lbrack
\alpha,\infty\lbrack}\left(  t\right)
\]
is positive. It follows that%
\[
kP_{\left(  x-\alpha e\right)  ^{+}}x\leq P_{\left(  x-\alpha e\right)  ^{+}%
}\varphi\left(  x\right)  \leq\varphi\left(  x\right)  .
\]
Using this we obtain
\[
kTP_{\left(  x-\alpha e\right)  ^{+}}x\leq P_{\left(  x-\alpha e\right)  ^{+}%
}\varphi\left(  x\right)  \leq T\varphi\left(  x\right)  .
\]
Therefore we have%
\[
\sup\limits_{\beta\geq\alpha}\sup\limits_{x\in\mathcal{F}}TP_{\left(  x-\alpha
e\right)  ^{+}}x\leq k^{-1}y
\]
where $y=\sup\limits_{x\in\mathcal{F}}T\varphi\left(  x\right)  .$ Hence
\[
\lim\limits_{\alpha\longrightarrow\infty}\sup\limits_{x\in\mathcal{F}%
}TP_{\left(  x-\alpha e\right)  ^{+}}x=0,
\]
which shows that $\mathcal{F}$ is $T$-uniform.
\end{proof}

\begin{corollary}
\cite[Corollary 3.10]{L-180}Let $\mathcal{F}$ be a family in $L^{1}\left(
T\right)  $ and $1<p<\infty.$ If $\mathcal{F}$ is bounded in $L^{p}\left(
T\right)  ,$ that is, $T\left\vert x\right\vert ^{p}\leq y$ for some $y\in
L^{p}\left(  T\right)  ,$ for all $x\in\mathcal{F},$ then $\mathcal{F}$ is $T$-uniform.
\end{corollary}

\section{Strong sequential completeness}

We will consider in this section a conditional Riesz triple $\left(
E,e,T\right)  $ as above. Recall that a net $\left(  x_{\alpha}\right)  $ in
$E$ is $T$-strongly convergent to $x$ if $x_{a}$ converges to $x$ with respect
to the $R\left(  T\right)  $-vector valued norm, that is, $T\left\vert
x_{\alpha}-x\right\vert \longrightarrow0$ in order. We define in a similar way
the notion of strongly Cauchy nets. We say that $E$ is strongly complete if
every strongly Cauchy net is strongly convergent. Similarly, we define
sequential strong completeness for sequences. In \cite{L-360} the authors
proved that the space $L^{1}\left(  T\right)  $ is sequentially strong
complete, that is, every strongly Cauchy sequence $\left(  f_{n}\right)  $ in
$L^{1}\left(  T\right)  $ is strongly convergent to some $f\in L^{1}\left(
T\right)  .$ The proof is not an adaptation of the classical case, but it is
also not feasible to adapt it to the case of $L^{p}\left(  T\right)  $ for
$p\in\left(  1,\infty\right)  .$ In this section, we will demonstrate how the
application of $T$-uniformity enables us to establish that the spaces
$L^{p}\left(  T\right)  $ are sequentially strongly complete.

The complete understanding of full completeness for nets remains unknown.
However, we do know that the space $L^{\infty}\left(  T\right)  $ is strongly
complete, and in this particular case, the analysis is relatively
straightforward (see \cite{L360}). To prove our result for $L^{p}\left(
T\right)  $ we will combine the result of Kuo, Rodda and Watson about
$L^{1}\left(  T\right)  $ and a generalization of Vitali Theorem in Riesz
spaces$.$ We need also another kind of convergence that generalizes the
convergence in probability. Its definitions and some of its basic properties
are presented in the second section. As every sequence is trivially locally
bounded, the following result is an immediate consequence of \cite[Theorem
5.5]{L-14}.

\begin{theorem}
\label{Th0}Let $\left(  E,e,T\right)  $ be a conditional Riesz triple and
$\left(  f_{n}\right)  $ a sequence in $L^{1}\left(  T\right)  .$ Then the
following statements are equivalent.

\begin{enumerate}
\item[(i)] The sequence $\left(  f_{n}\right)  $ converges $T$-strongly to
$f;$

\item[(ii)] The sequence $\left(  f_{n}\right)  $ is $T$-uniform and converges
to $f$ in $T$-conditional probability.
\end{enumerate}
\end{theorem}

In \cite{L-180} we obtain a generalization of \cite[Theorem 5.5]{L-14} for
$L^{p}\left(  T\right)  .$ We state here a version for sequences and provide a
proof for sake of completeness (also some details in the proof in \cite{L-180}
need to be clarified).

\begin{theorem}
\label{Th1}Let $\left(  f_{n}\right)  $ be a sequence in $L^{p}\left(
T\right)  $ where $1\leq p<\infty$ then the following statements are equivalent.

\begin{enumerate}
\item[(i)] The sequence $\left(  f_{n}\right)  $ converges to $f$ in
$L^{p}\left(  T\right)  $ (with respect to the norm $\left\Vert .\right\Vert
_{p,T}$)$;$

\item[(ii)] The sequence $\left(  \left\vert f_{n}\right\vert ^{p}\right)  $
is $T$-uniform and converges to $f$ in $T$-conditional probability.
\end{enumerate}
\end{theorem}

\begin{proof}
(i) $\Longrightarrow$ (ii) By Lyapunov inequality \cite[Corollary 3.8]{L-180},
$T\left\vert f_{n}-f\right\vert \leq\left(  T\left\vert f_{n}-f\right\vert
^{p}\right)  ^{1/p}$ and so $f_{n}$ converges strongly to $f$ in $L^{1}\left(
T\right)  .$ According to Theorem \ref{Th0} the sequence $\left(
f_{n}\right)  $ converges to $f$ in $T$-conditional probability. To show that
$\left(  \left\vert f_{n}\right\vert ^{p}\right)  $ is $T$-uniform let us
consider a net $\left(  P_{\alpha}\right)  $ of band projections such that
$P_{\alpha}\overset{o}{\longrightarrow}0$ and it will be enough by
\cite[Theorem 4.2]{L-14} to show that $\lim\limits_{\alpha}\sup\limits_{n}%
TP_{\alpha}\left\vert f_{n}\right\vert ^{p}=0$. Observe first that%
\[
TP_{\alpha}\left\vert f_{n}\right\vert ^{p}\leq2^{p-1}TP_{\alpha}\left\vert
f_{n}-f\right\vert ^{p}+2^{p-1}TP_{\alpha}\left\vert f\right\vert ^{p}.
\]
Thus we have for every $k\in\mathbb{N},$%
\[
\sup\limits_{n}TP_{\alpha}\left\vert f_{n}\right\vert ^{p}\leq TP_{\alpha}%
\sup\limits_{n\leq k}\left\vert f_{k}\right\vert ^{p}+2^{p-1}TP_{\alpha
}\left\vert f\right\vert ^{p}+2^{p-1}\sup\limits_{n\geq k}T\left\vert
f_{n}-f_{k}\right\vert ^{p}.
\]
\newline By taking the $\lim\sup$ over $\alpha$ we obtain%
\[
\limsup\limits_{\alpha}\sup\limits_{n}TP_{\alpha}\left\vert f_{n}\right\vert
^{p}\leq2^{p-1}\sup\limits_{n\geq k}T\left\vert f_{n}-f_{k}\right\vert ^{p}.
\]
As this happens for every $k$ we get $\limsup\limits_{\alpha}\sup
\limits_{n}TP_{\alpha}\left\vert f_{n}\right\vert ^{p}=0$ and we are done.

(ii) $\Longrightarrow$ (i) We show first that $f$ belongs to $L^{p}\left(
T\right)  .$ As $f_{n}\longrightarrow f$ in $T$-conditional probability we
have $\left\vert f_{n}\right\vert \longrightarrow\left\vert f\right\vert $ and
$\left\vert f_{n}\right\vert ^{p}\longrightarrow\left\vert f\right\vert ^{p}$
in $T$-conditional probability as well. Hence for every $k,$ $\left(
T\left\vert f_{n}^{p}\right\vert \wedge ke\right)  $ converges in order to
$T\left\vert f\right\vert ^{p}\wedge ke$ as $n\longrightarrow\infty.$ Whence
It follows that $T\left\vert f\right\vert ^{p}\wedge ke\leq M:=\sup
T\left\vert f_{n}\right\vert ^{p}.$ Now as $L^{1}\left(  T\right)  $ is
$T$universally complete $\left\vert f\right\vert ^{p}\in L^{1}\left(
T\right)  $ and then $f\in L^{p}\left(  T\right)  .$ Now the inequality%
\[
\left\vert f_{n}-f\right\vert ^{p}\leq2^{p-1}\left(  \left\vert f_{n}%
\right\vert ^{p}+\left\vert f\right\vert ^{p}\right)
\]
shows that $\left(  \left\vert f_{n}-f\right\vert ^{p}\right)  $ is
$T$-uniform. Also we use the following obvious result, true for $p\geq1,$%
\[
\left\vert f_{n}-f\right\vert ^{p}\wedge e=\left(  \left\vert f_{n}%
-f\right\vert \wedge e\right)  ^{p}\leq\left\vert f_{n}-f\right\vert \wedge e.
\]
It follows from \cite[Lemma 4]{L-900} that $\left\vert f_{n}-f\right\vert
^{p}$ converges to $0$ in $T$-conditional probability. Now using Theorem
\ref{Th0} we deduce that $T\left\vert f_{n}-f\right\vert ^{p}\longrightarrow
0,$ and so $f_{n}\longrightarrow f$ in $L^{p}\left(  T\right)  $ as required.
\end{proof}

We are ready now to prove the main result of this section which answers
positively an open question asked in \cite{L-360}.

\begin{theorem}
\label{ThB}Let $\left(  E,e,T\right)  $ be a conditional Riesz triple. Then
the space $L^{p}\left(  T\right)  $ is sequentially strongly complete for
$1\leq p\leq\infty.$
\end{theorem}

\begin{proof}
(i) Assume that $1<p<\infty.$ Let $\left(  f_{n}\right)  $ be a Cauchy
sequence in $L^{p}\left(  T\right)  .$ Then by Lyapunov Inequallity
\cite[Corollary 3.8]{L-180}, $\left(  f_{n}\right)  $ is a Cauchy sequence in
$L^{1}\left(  T\right)  .$ It follows from the strong sequential completeness
of $L^{1}\left(  T\right)  $ \cite[Theorem ??]{L-360} that $\left(
f_{n}\right)  $ is strongly convergent in $L^{1}\left(  T\right)  $, that is,
there exists $f\in L^{1}\left(  T\right)  $ such that $T\left\vert
f_{n}-f\right\vert \longrightarrow0$ in order. According to Theorem \ref{Th1},
we deduce that $f_{n}\longrightarrow f$ in $T$-conditional probability.
According to Lemma \ref{UI-C} we deduce that $\left\vert f_{n}\right\vert
^{p}\longrightarrow\left\vert f\right\vert ^{p}$ in $T$-conditional
probability. It follows from Lemma \ref{UI-B} that
\[
T\left(  \left\vert f_{n}\right\vert ^{p}\wedge ke\right)  \overset
{o}{\longrightarrow}T\left(  \left\vert f\right\vert ^{p}\wedge ke\right)  ,
\]
for every integer $k.$ Thus%
\[
T\left\vert f\right\vert ^{p}\wedge ke\leq y:=\sup T\left\vert f_{n}%
\right\vert ^{p}\in L^{1}\left(  T\right)  .
\]
This shows that $f\in L^{p}\left(  T\right)  .$ To complete the proof we will
invoke Theorem \ref{Th1}, and it remains only to show that $\left(  \left\vert
f_{n}\right\vert ^{p}\right)  $ is $T$-uniform. To this end consider a net
$\left(  P_{\alpha}\right)  $ of band projections such that $P_{\alpha
}\overset{o}{\longrightarrow}0$ and let $k$ be fixed in $\mathbb{N}.$ Then%
\[
TP_{\alpha}\left\vert f_{n}\right\vert ^{p}\leq2^{p-1}TP_{\alpha}\left\vert
f_{n}-f_{k}\right\vert ^{p}+2^{p-1}TP_{\alpha}\left\vert f_{k}\right\vert
^{p}.
\]
Thus%
\[
\sup\limits_{n}TP_{\alpha}\left\vert f_{n}\right\vert ^{p}\leq TP_{\alpha}%
\sup\limits_{k\leq n}\left\vert f_{k}\right\vert ^{p}+2^{p-1}TP_{\alpha
}\left\vert f_{k}\right\vert ^{p}+2^{p-1}\sup\limits_{n\geq k}T\left\vert
f_{n}-f_{k}\right\vert ^{p}.
\]
It follows that%
\[
\limsup\limits_{\alpha}\sup\limits_{n}TP_{\alpha}\left\vert f_{n}\right\vert
^{p}\leq2^{p-1}\sup\limits_{n\geq k}T\left\vert f_{n}-f_{k}\right\vert .
\]
As this happens for every integer $k$ we get%
\[
\limsup\limits_{\alpha}\sup\limits_{n}TP_{\alpha}\left\vert f_{n}\right\vert
^{p}=0.
\]
\newline This shows that the sequence $\left(  \left\vert f_{n}\right\vert
^{p}\right)  $ is $T$-uniform and completes the proof of this part.

(ii) We are almost done because the case $p=1$ is already proved in
\cite{L-360} and for the case $p=\infty$ we have even a stronger result
obtained also in \cite{L-360} stating that $L^{\infty}\left(  T\right)  $ is
rather strong complete. But as the proof there seems to be somehow quick we
provide here the full details for the strong sequential completeness of
$L^{\infty}\left(  T\right)  $. Recall that%
\[
\left\Vert f\right\Vert _{T,\infty}=\inf\left\{  u\in R\left(  T\right)
:\left\vert f\right\vert \leq u\right\}  .
\]
Let $\left(  f_{n}\right)  $ be a Cauchy sequence in $L^{\infty}\left(
T\right)  $ and put%
\[
h_{n}=\sup\limits_{r,s\geq n}\left\Vert f_{r}-f_{s}\right\Vert _{T,\infty}\in
L^{\infty}\left(  T\right)  .
\]
Then $h_{n}\downarrow0.$ It follows easily from the $T$universal completeness
of $L^{1}\left(  T\right)  $ that $R\left(  T\right)  $ is a regular Riesz
subspace of $L^{1}\left(  T\right)  $ (see also Theorem 6 in \cite{L-900} and
the remark that follows). Also $L^{\infty}\left(  T\right)  $ is regular
because it is an ideal in $L^{1}\left(  T\right)  $ (\cite{L-65}). Let
$\left(  f_{n}\right)  $ be a Cauchy sequence in $L^{\infty}\left(  T\right)
$ and let $v_{n}=\sup\limits_{p,q\geq n}\left\Vert f_{p}-f_{q}\right\Vert
_{\infty,T}.$ Then $v_{n}\downarrow0$ in $L^{\infty}\left(  T\right)  $ so
$v_{n}\downarrow0$ in $L^{1}\left(  T\right)  .$ Now as $\left\vert
f_{p}-f_{q}\right\vert \leq\left\Vert f_{p}-f_{q}\right\Vert _{\infty,T}$ we
get $\sup\limits_{p,q\geq n}\left\vert f_{p}-f_{q}\right\vert \downarrow0$ in
$L^{1}\left(  T\right)  .$ But $L^{1}\left(  T\right)  $ is order complete and
so there exists $f\in L^{1}\left(  T\right)  $ such that $\left\vert
f_{n}-f\right\vert \overset{o}{\longrightarrow}0$. Now letting
$p\longrightarrow\infty$ in the inequality%
\[
\left\vert f_{n}-f_{p}\right\vert \leq\left\Vert f_{n}-f_{p}\right\Vert
_{\infty,T}\leq v_{n},
\]
we get $\left\vert f_{n}-f\right\vert \leq v_{n}\in R\left(  T\right)  $ and
so $\left\Vert f_{n}-f\right\Vert _{\infty,T}$ $\overset{o}{\longrightarrow
}0.$ This completes the proof.
\end{proof}

\end{document}